\documentclass[a4paper, fullpage, reqno, 11pt]{amsart}
\usepackage[english]{babel}
\usepackage{amsmath}
\usepackage{amssymb}
\usepackage{enumerate}
\usepackage{a4wide}
\usepackage{pdfsync}
\usepackage{ifthen} 

\usepackage{pdfsync}
\provideboolean{shownotes} 
\setboolean{shownotes}{true} 
\newcommand{\margnote}[1]{
\ifthenelse{\boolean{shownotes}}%
{\marginpar{\raggedright\tiny\texttt{#1}}}%
{}%
}

\newcommand{\hole}[1]{
\ifthenelse{\boolean{shownotes}}%
{\begin{center} \fbox{ \rule {.25cm}{0cm}
\rule[-.1cm]{0cm}{.4cm} \parbox{.85\textwidth}{\begin{center}
\texttt{#1}\end{center}} \rule {.25cm}{0cm}}\end{center}}
{}
}
\newtheorem{theorem}{Theorem}[section]

\newtheorem{lemma}[theorem]{Lemma}

\theoremstyle{remark}

\newtheorem{definition}[theorem]{Definition}

\newcommand{\e}{\varepsilon}		       
\newcommand{\R}{\mathbb{R}}

\newcommand{\ue}{u^{\varepsilon}}

\newcommand{\pe}{p^{\varepsilon}}

\newcommand{\dive}{\mathop{\mathrm {div}}}
\newcommand{\PH}{\dot{H}}

 \def\rest#1#2{{#1}\lfloor\lower3pt\hbox{\kern-2pt\_\kern-1pt\_}{#2}}

\numberwithin{equation}{section}

\begin{document}
\title[Artificial compressibility and suitable weak solutions]
{Weak solutions of Navier-Stokes equations\\
 constructed by artificial compressibility method\\
  are suitable }

\author{DONATELLA DONATELLI}

\address{Dipartimento di Matematica Pura ed Applicata,\\
Universit\`a degli Studi dell' Aquila\\
67100 L'Aquila,Italy.}
\email{donatell@univaq.it }

\author{STEFANO SPIRITO}

\address{Dipartimento di Matematica Pura ed Applicata,\\
Universit\`a degli Studi dell' Aquila\\
67100 L'Aquila,Italy.}
\email{stefano.spirito@dm.univaq.it }

\begin{abstract}
In this paper we prove that weak solution constructed by artificial compressibility method are suitable in the sense of Scheffer, \cite{sv}, \cite{sv1}.
Using Hilbertian setting and Fourier transform with respect to the time we obtain nontrivial estimates of the pressure and the time derivate which allow us to pass into the limit.
\end{abstract}

\maketitle
\section{Introduction}\label{intro}
In this paper we investigate if the weak solutions of the Navier-Stokes equations are suitable in the sense of Scheffer \cite{sv}, \cite{sv1}. 
The incompressible Navier-Stokes equations in three spatial dimensions with unit 
viscosity and zero external force are given by  the following system
\begin{equation}
 \begin{cases}
 \displaystyle{\partial_{t}u-\Delta u+(u\cdot\nabla)u+\nabla p=0}\\
 \dive u=0,
 \end{cases}
 \label{1.1}
 \end{equation}
 where $(x,t)\in\R^{3}\times [0,T]$, $u\in\R^{3}$ denotes the velocity vector field and $p\in \R$ the pressure of the fluid. Let us recall the notion of Leray weak solution. 
 \begin{definition}
 We say that $u\in L^{\infty}((0,T);L^{2}(\R^3))\cap L^2((0,T);\dot{H}^{1}(\R^3)) $ is a Leray weak solution of the Navier-Stokes equations if it satisfies \eqref{1.1} in the sense of distribution for all
$\psi\in C^{\infty}_{0}(\R^3 \times\R)$ with $\dive \psi=0$ and moreover the following energy inequality holds for every $t\in[0,T]$
\begin{equation}
\int_{\R^{3}}dx |u(x,t)|^{2}+2\int_{0}^{t}ds\int_{\R^{3}}dx |\nabla u(x,s)|^{2} \leq \int_{\R^{3}}dx |u(x,0)|^{2}.
\label{Ler}
\end{equation}
 \end{definition}
In the mathematical literature there exists several proofs of the global existence of Leray weak solution for divergence-free initial data in $L^{2}(\R^{3})$, see for example the book of P. L. Lions \cite{LPL96} or the monograph 
of T{\'e}mam \cite{Tem01}. Several problems about the Leray weak solutions are still open, for example it is not know whether or not the solutions are unique or develop singularities in a finite time also for smooth initial data.  
In fact the regularity requirements for proving uniqueness and global in time regularity are not available at the present time for Leray weak solutions.
Scheffer in \cite{sv}, \cite{sv1} introduced the notion of suitable weak solutions that we recall here.
\begin{definition}
 Let $(u,p)$, $u\in L^{2}((0,T);H^{1}(\R^{3}))\cap L^{\infty}((0,T);L^{2}(\R^{3}))$, $p\in {\mathcal{D}}^{'}((0,T);L^{2}(\R^3))$, be a weak solution to the Navier-Stokes equation \eqref{1.1}.
 The pair $(u,p)$ is said a suitable weak solutions if the following local energy balance
\begin{equation}
\partial_{t}\left(\frac{1}{2}|u|^{2}\right)+\nabla\cdot\left(\left(\frac{1}{2}|u|^{2}+p\right)u\right)-\Delta\left(\frac{1}{2}|u|^{2}\right)+ |\nabla u|^{2}-f\cdot u \leq 0
\label{1.2}
\end{equation}
holds in the distributional sense. 
 \end{definition}
In mathematical literature \eqref{1.2} is called  \emph{generalized energy inequality}.\\
It is important to remark that weak solutions constructed by Leray are suitable. At the moment the best regularity result for the weak solution of the Navier-Stokes equations is a partial regularity result, i.e. the so called Caffarelli-Kohn-Nirenberg theorem,  \cite{CKN}. This theorem asserts  that the one-dimensional parabolic Hausdorff measure of the singular set is zero and holds only for suitable weak solutions. If the class of suitable weak solutions is a proper class of Leray weak solution is an open problem since Scheffer's works. The method usually used to construct suitable weak solutions are regularization of the non-linear term , \cite{CKN}, adding hyper viscosity, \cite{BDV1}. Recently Guermond in  \cite{Guer}  proved
  that some type of Faedo-Galerkin approximation lead to a suitable weak solution. He used classical Lions \cite{L-JL59} method of fractional derivate in order to obtain maximal regularity in negative Sobolev space. In particular Guermond proved that Faedo-Galerkin weak solutions of the three dimensional Navier-Stokes equations with Dirichlet boundary condition are suitable provided they are constructed using finite-dimensional approximation spaces having a discrete commutator property. 
 
  In this paper we prove that Leray weak solutions costructed by the artificial compressibility method are suitable in the sense of the Definition 1.2.
  The artificial compressibility approximation was introduced by Chorin \cite{Ch68, Ch69}, Oskolkov \cite{Osk} and T{\'e}mam \cite{Tem69a, Tem69b}, in order to deal with the difficulty induced by the incompressibility constraints in the numerical approximation. The approximation system reads as follows
  \begin{equation}
\begin{cases}
\displaystyle{\partial_{t}\ue+\nabla \pe=\Delta \ue-\left(\ue\cdot\nabla\right)\ue -\frac{1}{2}(\dive \ue)\ue}\\
\e\partial_{t}\pe+ \dive \ue=0,
\end{cases}
\label{1.3}
\end{equation} 
 where $(x,t)\in \R^{3}\times [0,T]$, $\ue=\ue(x,t)\in\R^{3}$  and $\pe=\pe(x,t)\in \R$, $f^{\e}=f^{\e}(x,t)\in \R^{3}$.\\
Temam showed the convergence of this approximation on bounded domain. Recently in \cite{DM01} and \cite{DM02} the result was extended in the case of the whole space and exterior domain, respectively. In \cite{DM01} the convergence towards Leray weak solutions of the Navier-Stokes equations is achieved by using the dispersive structure of the system.
Here we will consider the system \eqref{1.3} endowed with the following initial conditions
 \begin{equation}
\ue(x,0)=\ue_{0}(x) \qquad\pe(x,0)=\pe_{0}(x), 
\label{1.4}
\end{equation} 
 such that
\begin{align} 
\ue_{0}&\rightarrow u_{0}\textrm{ in }L^{2}(\R^{3})\qquad\textrm{as $\e\rightarrow 0$}\label{1.5},\\
\sqrt{\varepsilon}\pe_{0}&\rightarrow 0\textrm{ in }L^{2}(\R)\qquad\textrm{as $\e\rightarrow 0$}.\label{1.6}
\end{align}
As we will see later on in Section 3, in order to get some a priori estimate on $\ue$ and $\pe$ we need to assume that
\begin{equation}
\ue_{0}\in{\PH}^{1}({\R}^{3}).
\label{plus3}
\end{equation}
We will be able to prove the following theorem
\begin{theorem}\label{main} 
Let $(\ue,\pe)$ be a weak solution of the system \eqref{1.3} with initial data \eqref{1.4} and such that \eqref{1.5}, \eqref{1.6} and \eqref{plus3} are satisfied.
Then $(\ue,\pe)$ converges to a suitable weak solution of the Navier-Stokes system as $\e$ goes to zero. 
\end{theorem}
In order to prove this theorem we have to estimate carefully the pressure term $\pe$ and $\sqrt{\e}\pe$. Since we haven't the incompressibility constraint we cannot use classical method based on the elliptic equation associated to the pressure.  Also the dispersive approach, as in \cite{DM01}, doesn't give usefull estimates. We will use the method of Lions of the fractional derivates for getting the necessary estimates.
This paper is organized as follows. In Section 2 we recall some basic facts about Navier-Stokes equations and the artificial compressibility method and fix some notations. In Section 3 we obtain the \emph{a priori} estimates which allow us to pass into the limit. In Section 4 we give the proof of the main result. We want to point out that the estimates of the Lemma 3.4 are the core of this paper.  


\section{Notations and preliminary results}

 For convenience of the reader we establish some notations and recall some basic fact about the Navier-Stokes equations.


 \subsection{Notations}
We will denote by  $\mathcal{D}(\R ^d \times \R_+)$
 the space of test function
$C^{\infty}_{0}(\R^d \times \R_+)$, by $\mathcal{D}'(\R^d \times
\R_+)$ the space of Schwartz distributions and $\langle \cdot, \cdot \rangle$
the duality bracket between $\mathcal{D}'$ and $\mathcal{D}$ and by $\mathcal{M}_{t}X'$ the space $C_{c}^{0}([0,T];X)'$. The inner product in $L^{2}(\R^{d})$ will be denoted by parentheses, 
namely $(u,v):=\int_{\R^{d}}u(x)\overline{v(x)}dx$. Moreover $W^{k,p}(\R^{d})=(I-\Delta)^{-\frac{k}{2}}L^{p}(\R^{d})$ and $H^{k}(\R^{d})=W^{k,2}(\R^{d})$ denote the non homogeneous Sobolev spaces for any $1\leq p\leq \infty$ and $k\in \R$. $\dot W^{k,p}(\R^{d})=(-\Delta)^{\frac{k}{2}}L^{p}(\R^{d})$ and $\dot H^{k}(\R^{d})=W^{k,2}(\R^{d})$  denote the homogeneous Sobolev spaces. \\
Let $H$ be a Hilbert space,  $\delta>1$, $L^{\delta}(\mathbb{R};H)$ is the Lebesgue space of function with values in $H$, namely
 \begin{equation*}
L^{\delta}(\mathbb{R};H):= \left\{\psi : \mathbb{R} \ni t\mapsto
\psi(t) \in H; \int_\mathbb{R} ||\psi(t)||^{\delta}_{H} dt <
\infty \right\}.
\end{equation*}
For any $\psi \in L^{1}(\mathbb{R};H)$ and  $k \in \mathbb{R}$ we define the Fourier transform of $\psi$ with respect to time as follows
\begin{equation*}
\tilde{\psi}(k)= \int_{\mathbb{R}} dt \psi(t) e^{-2i\pi kt}.
\end{equation*}
The previous definition can be extended to the space of tempered distribution $S^{'} (\mathbb{R};H)$ in the usual way.
We define $H^{\gamma}(\mathbb{R};H)$ the space of tempered distributions $v \in S^{'}(\mathbb{R};H)$ such that
\begin{equation*}
\int_{\mathbb{R}} (1 + |k|)^{2\gamma} ||\tilde{v}||^{2}_{H} dk < +
\infty.
\end{equation*}
The space $H^{\gamma}((0,T);H)$ is defined by those distributions that can be extended to $S^{'}(\mathbb{R};H)$ and whose extension is in $H^{\gamma}((0,T);H)$. The norm in   $H^{\gamma}((0,T);H)$
is the quotient norm, namely, 
\begin{equation*}
||v||_{H^{\gamma}((0,T);H)} = \inf_{ \begin{array}{cc}
                                \scriptstyle \widetilde{v}=u \\
                                \scriptstyle \textrm{q.o. su } (0,T)
                                \end{array}} ||{v}||_{H^{\gamma}(\mathbb{R};H)}.
\end{equation*}
The notations  $L^{p}(L^{q})$, $L^{p}(W^{k,q})$, $L^{p}(H^{s})$ and $H^{r}(H^{s})$ will abbreviate respectively  the spaces $L^{p}([0,T];L^{q}(\R^{d}))$,  $L^{p}([0,T];W^{k,q}(\R^{d}))$, 
$L^{p}([0,T];H^{s}(\R^{d}))$ and $H^{r}([0,T];H^{s}(\R^{d}))$. Moreover, we shall denote by $Q$ and $P$ respectively  the Leray's projectors $Q$ on the space of gradients vector fields and $P$ on the space of divergence - free vector fields. Namely
\begin{equation}
Q=\nabla \Delta^{-1}\dive\qquad P=I-Q.
\label{2.1}
\end{equation} 
Let us remark that   $Q$ and $P$  can be expressed in terms of Riesz multipliers, therefore they are  bounded linear operators on every $W^{k,p}$ $(1< p<\infty)$ space (see \cite{Ste93}).   \\ \\


\subsection{Compactness lemma}

We recall some Aubin-Lions like compactness results. We will use these lemmas in the proof of Theorem 1.1.

\begin{lemma} \label{C1}
Let $H_{0}\subset H\subset H_{1}$ be three Hilbert space with dense and continuos embedding. Assume the the embedding $H_{0}\subset H$ is compact and
let $\gamma>0$ be a positive real number. Then, the injection $L^{2}((0,T);H_{0})\cap H^{\gamma}((0,T);H_{1})\rightarrow L^{2}((0,T);H)$
is compact.
\end{lemma}
\begin{proof}
See Lions \cite{L-JL69} p.61 Theorem 5.2
\end{proof}

\begin{lemma}\label{C2}
Let $X\subset Y$ be two Hilbert space with compact embedding and let $\tau>\frac{1}{2}$. The injection $H^{\tau}((0,T);X)\rightarrow C^{0}([0,T];X)$ is continuos and the injection 
$H^{\tau}((0,T);X)\rightarrow C^{0}([0,T];Y)$
 is compact.
 \end{lemma}
\begin{proof}
See \cite{Guer}, Appendix A.1.
\end{proof}

\begin{lemma}\label{C3}
Let $H_{0}\subset H_{1}$ be two Hilbert spaces with compact embedding. Let $\gamma>0$ and $\gamma>\nu$, then the injection $H^{\gamma}((0,T);H_{0})\subset H^{\nu}((0,T);H_{1})$ is compact.
\end{lemma}
\begin{proof}
See \cite{Guer} Appendix A.2
\end{proof}


\subsection{Artificial Compressibility Approximation}

 In this section we recall some previous result about the approximating system that we rewrite for convenience of the reader.
 \begin{equation}
\begin{cases}
\displaystyle{\partial_{t}\ue+\nabla \pe=\mu\Delta \ue-\left(\ue\cdot\nabla\right)\ue -\frac{1}{2}(\dive \ue)\ue}\\
\e\partial_{t}\pe+ \dive \ue=0,
\end{cases}
\label{3.1}
\end{equation} 
As said in the Introduction we consider  two initial condition, namely
 \begin{equation}
 \ue(x,0)=\ue_{0}(x)\qquad\pe(x,0)=\pe_{0}(x),
 \label{3.2}
  \end{equation}
such that
\begin{align}
\ue_{0}\rightarrow u_{0}\textrm{ in }L^{2}(\R^{3})\qquad\textrm{as $\e\rightarrow 0$}
\label{3.3}\\
\sqrt{\varepsilon}\pe_{0}\rightarrow 0\textrm{ in }L^{2}(\R)\qquad\textrm{as $\e\rightarrow 0$}.
\label{3.4}
\end{align}
We have an existence theorem for the system \eqref{3.1} with the above initial data.
\begin{theorem}\label{ex}
Let $(\ue_{0},\pe_{0})$ satisfy the conditions \eqref{3.3} and \eqref{3.4}, $\varepsilon>0$. Then the system \eqref{3.1} has a weak solution $(\ue,\pe)$ with the following properties
\begin{itemize}
\item $\ue\in L^{\infty}((0,T);L^{2}(\R^3))\cap L^{2}((0,T);{H}^{1}(\R^{3}))$
\item $\sqrt{\varepsilon}\pe\in L^{\infty}((0,T);L^{2}(\R^{3}))$
\end{itemize}
for all $T>0$.
\end{theorem}
The proof of this theorem can be omitted since it is a direct consequence of all the a priori bounds in  \cite{DM01} and it follows by using standard finite dimensional Galerkin type approximations.
 In \cite{DM01} the following result was showed.
 \begin{theorem}\label{MDT}
Let $(\ue,\pe)$ be a sequence of weak solutions in $\R^{3}$ of the system \eqref{3.1}, assume that the initial data satisfy \eqref{3.3} and \eqref{3.4}. Then 
\begin{itemize} 
\item{}There exist $u\in L^{\infty}((0,T);L^{2}(\R^3))\cap L^{2}((0,T);{\PH}^{1}(\R^{3}))$ such that
\begin{equation}
\ue\rightarrow u\textrm{ weakly in } L^{2}((0,T);{\PH}^{1}(\R^{3})).
\label{3.5}
\end{equation}
\item{}For any $p\in[4,6)$ the gradient component $Q\ue$ of the vector field $\ue$ satisfies
\begin{equation}
 Q\ue\rightarrow 0\textrm{ strongly in }L^{2}((0,T);L^{p}(\R^{3})).
 \label{3.6}
 \end{equation}
 \item{} The divergence-free component $P\ue$ of the vector field $\ue$ satisfies
 \begin{equation}
 P\ue\rightarrow Pu=u\textrm{ strongly in }L^{2}((0,T);L^{2}_{loc}(\R^{3})).
 \label{3.7}
 \end{equation}
 \item{}The sequence $\{\pe\}$ converges in the sense of the distributions to
 \begin{equation}
 p=\Delta^{-1}\left(tr\left((Du)^{2}\right)\right) 
 \label{3.8}
 \end{equation}
 \end{itemize}
 Moreover $Pu=u$ is a Leray weak solution to the incompressible Navier-Stokes equation
 \begin{equation}
 P\left(\partial_{t}u-\Delta u+(u\cdot\nabla)u\right)=0\textrm{ in }D^{'}([0,T]\times \R^{3})
 \end{equation} 
 and the following inequality holds for every $t\in[0,T]$,
 \begin{equation}
 \int_{\R^{3}}dx |u(x,t)|^{2}+2\int_{0}^{t}ds\int_{\R^{3}}dx |\nabla u(x,s)|^{2} \leq \int_{\R^{3}}dx |u(x,0)|^{2}.
\label{3.9}
\end{equation}
\end{theorem}
It is worth to mention here that in order to prove Theorem 2.4 and Theorem 2.5 only the conditions  \eqref{3.3} and \eqref{3.4} are needed. We don't need to assume $\ue_{0}\in{\PH}^{1}$, namely \eqref{plus3}.

\section{A Priori Estimates}

In this section we deal with the \emph{a priori} bounds for our approximating system. \\
By taking into account the Corollary 4.2 and Theorem 4.3 of \cite{DM01} we recall the following estimates.
\begin{lemma}\label{mardon}
Let $(\ue\pe)$ be a weak solution of the system \eqref{3.1} with initial data \eqref{3.2}. Assume that the conditions \eqref{3.3} and \eqref{3.4} hold. Then there exists $c>0$, indipendent on $\e$, such that the follwing estimates hold
\begin{align*} 
& \sqrt{\e}\pe &\quad  \text{is bounded in $L^{\infty}([0,T];L^{2}(\R^{3}))$,}\\
& \e\pe_{t} &\quad  \text{is relatively compact in $H^{-1}([0,T]\times \R^{3}),$}\\
& \nabla\ue &\quad  \text{is bounded in $L^{2}([0,T]\times\R^{3}),$}\\
& \ue &\quad  \text{is bounded in $L^{\infty}([0,T];L^{2}(\R^{3}))\cap L^{2}([0,T];L^{6}(\R^{3})),$}\\
(&\ue \!\cdot\!\nabla)\ue &\quad  \text{is bounded in $L^{2}([0,T];L^{1}(\R^{3}))\cap L^{1}([0,T];L^{3/2}(\R^{3})),$}\\
(& div\ue)\ue &\quad  \text{is bounded in $L^{2}([0,T];L^{1}(\R^{3}))\cap L^{1}([0,T];L^{3/2}(\R^{3})).$}\\
&\e^{3/8}\pe &\quad \text{is bounded in $L^{4}([0,T];W^{-2,4}(\R^{3}))$}\\
&\e^{7/8}\partial_{t}\pe &\quad \text{is bounded in $L^{4}([0,T];W^{-3,4}(\R^{3}))$}
\end{align*}
\end{lemma}
Unfortunately these estimates are not sufficient to prove the generalized energy inequality \eqref{1.2}. The aim of this section is to estimate more carefully  the nonlinear term and to get estimates on Hilbert space for the pressure 
$\pe$ and the term $\sqrt{\e}\pe$.\\
 Let be $p,q,\bar{r}$ and $s$  real numbers such that the following relations holds:
\begin{equation}
\frac{2}{p}+\frac{3}{q}=4,\quad p\in[1.2],\quad q\in[1,\frac{3}{2}],\quad \frac{s}{3}:=\frac{1}{q}-\frac{1}{2},\quad \bar{r}:=\frac{1}{p}-\frac{1}{2}.
\label{4.1}
\end{equation}
These relations follow from the Sobolev embedding theorems, in particular if $p,q$ satisfy \eqref{4.1} then the following embedding holds
\begin{equation}
L^{p}((0,T);L^{q}({\R}^{3}))\subset H^{-r}((0,T);H^{-s}({\R}^{3}))
\label{4.2}
\end{equation}
The first \emph{a priori} estimate regards the nonlinear terms
\begin{lemma}\label{NL}
Let $\ue$ be a weak solutions of the system \eqref{3.1}, then there exists a constat $c>0$, indipendent on $\e$, such that
\begin{equation}
|| (\ue\cdot\nabla)\ue+\frac{1}{2}\ue\dive\ue ||_{H^{-r}(\PH^{-s})}\leq c.
\label{4.3}
\end{equation}
\end{lemma}
The proof of \eqref{4.3} can be omitted since it follows by standard interpolation argument and the embedding \eqref{4.2}.
\subsection{Estimates of the velocity vector field}
In this section we get the estimate on $\ue$ and $\partial_{t}\ue$ on negative Sobolev spaces. So we need to assume
\begin{equation*}
\ue_{0}\in{\PH}^{1}(\R^{3}). 
\end{equation*}
Since we are going to use Fourier transform with respect to time we need to extend all the function from $[0,T]$ to $\R$.
We define by $\overline{\ue}$ the following extension of $\ue$
\begin{equation*}
\overline{\ue} = \left\{
                   \begin{array}{ccc}
                   (t+1)\ue_{0} & \textrm{ on } [-1,0] \\
                    \ue & \textrm{  } (0,
                   T+1) \\ 0 & \textrm{ on } [T+1,\infty].
                   \end{array}
                   \right.
\end{equation*}
Let  $\varphi \in C^{\infty}(\mathbb{R})$ be such that
$supp(\varphi)\subset(-1,T+1)$ and $\varphi\equiv 1$ on $[0,T]$, we denote with abuse of notation
\[
\ue=\varphi \overline\ue.
\]
Next we define the following function
\begin{equation*} f^{\e} = \left\{
                \begin{array}{cc}
                (1+t)\varphi^{'} \ue_{0}+\varphi \ue_{0}-(1+t)\varphi ((\ue\cdot\nabla)\ue+\frac{1}{2}\ue\dive\ue) & \textrm{  }t\in (-1,0) \\
                -\varphi((\ue\cdot\nabla)\ue+\frac{1}{2}\ue\dive\ue)+\varphi^{'}\ue & \textrm{  }t\not\in (-1,0)
                \end{array}
                \right.
\end{equation*}
It follows that $\ue$ e $f^{\e}$ are well definied on $(-\infty,+\infty)$. By using \eqref{plus3} it follows that there exist $c>0$, indipendent on $\e$, such that
\begin{equation}\label{4.4}
||f^{\e}||_{H^{-r}({\PH}^{-s})}
\leq c
\end{equation}

\begin{lemma}\label{Vel}
Let $(\ue,\pe)$ be a solution of the system \eqref{3.1} then, there exists $c>0$, independent on $\e$, such that 
\begin{itemize}
\item for all $s\in\left[\frac{1}{2},\frac{3}{2}\right)$ and $r>\bar{r}$, one has
\begin{align}
||\partial_{t}\ue||_{H^{-r}({\PH}^{-s})}&\leq c,
\label{4.5}\\
||\Delta\ue||_{H^{-r}({\PH}^{-s})}&\leq c,
\label{4.6}
\end{align}
\item  for all $\alpha\in\left[\frac{1}{4},\frac{1}{2}\right)$ and $\tau<\bar{\tau}=\frac{2}{5}(1+\alpha)$, one has
\begin{equation}
||\ue||_{H^{-\tau}({\PH}^{-\alpha})}\leq c.
\label{4.7}
\end{equation}
\end{itemize}
\end{lemma}
\begin{proof}
We start by proving \eqref{4.7}. By taking the Fourier transform with respect to time of the system \eqref{3.1} we obtain
\begin{equation}
\begin{cases} 
2i\pi k \tilde{\ue}-\Delta\tilde{\ue}+\nabla\tilde{\pe}=\tilde{f^{\varepsilon}} \label{4.8}\\
\varepsilon 2i\pi k \tilde{\pe}+\dive \tilde{\ue}=0 
\end{cases}
\end{equation}
Let $\alpha>0$, we multiply the first equation of \eqref{4.8} by the complex conjugate of $-\Delta^{-\alpha} \tilde{\ue}$ and the second by the complex conjugate
of $-\Delta^{-\alpha} \tilde{\pe}$. By summing up and by taking the imaginary part we obtain the following inequality
\begin{equation}
|k||(\tilde{\ue},\Delta^{-\alpha}\tilde{\ue})|\leq |(\tilde{f^{\varepsilon}},\Delta^{-\alpha}\tilde{\ue})|^{2}.
\label{4.9}
\end{equation}
Now we choose $\alpha$ such that $\alpha\leq s\leq 1 +2\alpha$. Since $s\in\left[\frac{1}{2},\frac{3}{2}\right)$ we have that 
$\alpha\in\left[\frac{1}{4},\frac{1}{2}\right)$. By interpolation we get
\begin{equation}
||\tilde{\ue}||_{{\PH}^s}\leq ||\tilde{\ue}||_{{\PH}^\alpha}^{\gamma}||\tilde{\ue}||_{{\PH}^{1+2\alpha}}^{1-\gamma},
\label{4.10}
\end{equation}
where $\gamma=\frac{2\alpha+1-s}{1+\alpha}$. Inserting \eqref{4.10} in \eqref{4.9} we have
\begin{equation}
|k|||\tilde{\ue}||_{H^{-\alpha}}^{2-\gamma}\leq c||\tilde{f_{\varepsilon}}||_{H^{-s}}||\tilde{\ue}||_{H^{1}}^{1-\gamma}.
\label{4.11}
\end{equation}
We set $\nu=\frac{2\gamma}{2-\gamma}$, then we have 
\begin{equation}
|k|^{\frac{2}{2-\gamma}-\nu}||\tilde{\ue}||_{H^{-\alpha}}^{2}\leq c(1+|k|)^{-\nu}||\tilde{f^{\varepsilon}}||_{H^{-s}}^{\frac{2}{2-\gamma}}||\tilde{\ue}||_{H^{1}}^{\frac{2(1-\gamma)}{2-\gamma}}.
\label{4.12}
\end{equation}
By integrating \eqref{4.12} with respect to the time and by using H\"{o}lder inequality we get
\begin{equation}
\int_{\R}dk|k|^{\frac{2}{2-\gamma}}||\tilde{\ue}||_{H^{-\alpha}}^{2}\leq c ||f^{\varepsilon}||_{H^{-r}(H^{-s})}^{\frac{2}{2-\gamma}}||\tilde{\ue}||_{L^{2}(H^{1})}^{\frac{2(1-\gamma)}{2-\gamma}}.
\label{4.13}
\end{equation}
If we set $\bar{\tau}:=\frac{1+\alpha}{1+s}(1-\bar{r})$ then \eqref{4.7} follows.\\
Now we are going to prove \eqref{4.5}. We multiply the equations of the system \eqref{4.8} respectively by the complex conjugate of
$-\Delta^{-s}\tilde{\ue}$ and $-\Delta^{-s}\tilde{\pe}$. By summing up and by taking the imaginary part we obtain
\begin{equation*}
|k|||\tilde{\ue}||_{H^{-s}}^{2}\leq c ||\tilde{f}^{\varepsilon}||_{H^{-s}}||\Delta^{-s}\tilde{\ue}||_{{\PH}^{s}}.
\end{equation*}
So we have
\begin{align}
\nonumber |k|||\tilde{\ue}||_{H^{-s}}&\leq c  ||\tilde{f}^{\varepsilon}||_{H^{-s}},\\
 ||\tilde{\partial_{t}\ue}||_{H^{-s}}&\leq c  ||\tilde{f}^{\varepsilon}||_{H^{-s}}. \label{4.14}
\end{align}
Thus \eqref{4.5} follows by integrating with respect to time \eqref{4.14}.
It remains only to prove \eqref{4.6}. Testing the first and the second equation of \eqref{4.8} with  $-\Delta^{1-s}\tilde{\ue}$ and with  $\Delta^{1-s}{\pe}$ respectively, we have
\begin{equation}
||\Delta\ue||_{H^{-s}}^{2}\leq c ||f^{\varepsilon}||_{H^{-s}}||\Delta\ue||_{H^{-s}} .
\label{plus}
\end{equation}
Thus by integrating in time \eqref{plus} we obtain \eqref{4.6}.
\end{proof}
\subsection{Estimates of the pressure}

\begin{lemma}\label{Pre}
Let $(\ue,\pe)$ be a weak solution of \eqref{3.1}, there exist a constant $c>0$, independent on $\e$, such that for $s\in\left[\frac{1}{2},\frac{3}{2}\right)$, $r>\bar{r}=\frac{3}{4}-\frac{s}{2}$, $\delta\in(0,\frac{1}{12})$ and $\beta\in\left(0,\frac{1-12\delta}{10}\right)$   we have
\begin{equation}
||\pe||_{H^{-r}({\PH}^{1-s})}\leq c 
\label{4.15}
\end{equation}
and
\begin{equation}
||\sqrt{\varepsilon}\pe||_{{\PH}^{\frac{1}{2}+\frac{\beta}{4}}(H^{-\frac{1}{2}+\delta})}\leq c
\label{4.16}
\end{equation}
\end{lemma}

\begin{proof}

The estimate \eqref{4.15} can be proved by observing that from the first equation of the system \eqref{3.1} we have
\begin{equation}
\nabla \pe=f^{\varepsilon}+\Delta\ue-\partial_{t}\ue
\label{4.17}
\end{equation}
and by using \eqref{4.5}, \eqref{4.6} and \eqref{4.4}.
Now we are going to prove  \eqref{4.16}. We have the following equation
\begin{equation}
2\pi ik\varepsilon\tilde{\pe}+\dive \ue=0
\label{4.19}
\end{equation}
By multiplying \eqref{4.19} by $\Delta^{-\frac{1}{2}+\delta}\tilde{\pe}$ we obtain
\begin{equation}
 |k|||\sqrt{\varepsilon}\tilde{\pe}||^{2}_{H^{-\frac{1}{2}+\delta}}\leq c |(\Delta^{-\frac{1}{2}+\delta}\ue,\nabla\pe)|.
 \label{4.20}
\end{equation}
From \eqref{4.20} we have 
\begin{equation}
|k|^{1+\frac{\beta}{2}} ||\sqrt{\varepsilon}\pe||^{2}_{H^{-\frac{1}{2}+\delta}}\leq (1+|k|)^{\frac{1}{2}+\beta} ||\tilde{\ue}||_{H^{-\frac{1}{2}+3\delta}}\frac{||\tilde{\pe}||_{{\PH}^{\frac{1}{2}-\frac{\delta}{2}}}}{(1+|k|)^{\frac{1}{2}+\frac{\beta}{2}}},
\label{4.21}
\end{equation}
where $\beta\in\left(0,\frac{1-12\delta}{10}\right)$. By integrating with respect to  time \eqref{4.21} and by using the H\"{o}lder inequality we have
\begin{equation}
||\sqrt{\varepsilon}\pe||_{{\PH}^{\frac{1}{2}+\frac{\beta}{4}}(H^{-\frac{1}{2}+\delta})}\leq ||\ue ||_{{\PH}^{\frac{1}{2}+\beta}(H^{-\frac{1}{2}+3\delta})}    ||\pe||_{H^{-\frac{1}{2}-\frac{\beta}{2}}({\PH}^{\frac{1}{2}-\delta})}.
\label{4.22}
\end{equation}
The right hand-side of \eqref{4.22} is bounded by using \eqref{4.7} and \eqref{4.15}, provided $\delta\in(0,\frac{1}{12})$ and $\beta\in\left(0,\frac{1-12\delta}{10}\right)$.
\end{proof}

\section{Convergence to a suitable weak solution}

In this section we give the proof of the theorem 1.1.\\ 
Let us multiply the first equation of \eqref{3.1} by $\ue\phi$ with $\phi\in C_{0}^{\infty}({\R}^{3}\times\R)$, $\phi >0$, then we have, 
\begin{equation}\label{5.1}
\int_{0}^{T}dt(\partial_{t}\ue,\ue\phi)-(\Delta \ue,\ue\phi)+(\ue\cdot\nabla\ue,\ue\phi)+\frac{1}{2}(\ue \dive \ue, \ue\phi)+(\nabla \pe,\ue\phi)=0.
\end{equation}
By integrating by parts we obtain
\begin{equation}
\int_{0}^{T}dt (|\nabla\ue|^{2},\phi)=\int_{0}^{T}dt\frac{|\ue|^{2}}{2}(\phi_{t}+\Delta\phi)+(\ue\frac{|\ue|^{2}}{2},\nabla\phi)+(\ue\pe,\nabla\phi)
+(\pe\dive\ue,\phi).
\label{5.2}
\end{equation}
We estimate each term of \eqref{5.2} separately.\\
By weak lower semicontinuity
and the fact that $\ue\rightharpoonup u$ weakly in $L^{2}({\PH}^{1})$ we have that 
\begin{equation}
\int_{0}^{T}dt (|\nabla u|^{2},\phi)\leq \liminf_{\varepsilon\rightarrow 0}\int_{0}^{T} dt (|\nabla \ue|^{2},\phi)
\label{5.3}
\end{equation}
Since $\ue\rightarrow u$ strongly in $L^{2}(L^{2}_{loc})$, we get 
\begin{equation}
\int_{0}^{T}\frac{|\ue|^{2}}{2}(\phi_{t}+\Delta\phi)\rightarrow\int_{0}^{T}dt \frac{|u|^{2}}{2}(\phi_{t}+\Delta\phi)\qquad\textrm{as $\e\rightarrow 0$}.
\label{5.4}
\end{equation}
Next, by interpolation we have that $\ue\rightarrow u$ strongly in $L^{2}(L^{3})$ and that $\ue$ is bounded in $L^{4}(L^{3})$, so it follows
\begin{equation}
\int_{0}^{T}dt (\ue\frac{|\ue|^{2}}{2},\nabla\phi)\rightarrow\int_{0}^{T}dt(u\frac{|u|^{2}}{2},\phi)\qquad\textrm{as $\e\rightarrow 0$}.
\label{5.5}
\end{equation}
In order to estimate the last two terms in \eqref{5.2} we have to use carefully the estimates of the Lemma 3.3 and Lemma 3.4. We start by estimating
\begin{equation*}
\int_{0}^{T}dt (\ue\pe,\nabla \phi).
\end{equation*}
Let $\eta>0$, we set $r=\frac{2}{5}+\eta$, $s=\frac{3}{10}$. This choice implies that $||\pe||_{H^{-r}(H^{s})}$ is uniformly bounded. So 
$\pe\rightharpoonup p$ weakly in $H^{-r}((0,T),H^{s}_{loc}({\R}^{3}))$. Now let ${\eta}^{'} \in [0,\frac{1}{20}]$ and set $\alpha=\frac{3}{10}-\eta^{'}\leq\frac{1}{4}$, $\tau=\frac{2}{5}(1+\alpha-{\varepsilon}^{'})$.
This choice implies that $||\ue||_{H^{\tau}(H^{-\alpha})}$ is uniformly bounded and as a consequence $\ue\rightharpoonup u$ weakly in $H^{\tau}((0,T);H^{-\alpha})$. By using Lemma 2.3 we obtain that $\ue\rightarrow u$ strongly in $H^{r}((0,T);H^{-s}_{loc}({\R}^{3}))$ provided $\tau>r$ and $s>\alpha$. Since ${\eta}^{'}\in(0,\frac{1}{20})$ we have $s>\alpha$, and if we choose $\eta=\frac{1}{50}(3-20{\eta}^{'})$ we have $\tau>r$. Now it follows easily that 
\begin{equation}
\int_{0}^{T}dt(\pe\ue,\nabla\phi)\rightarrow\int_{0}^{T}dt(pu,\nabla\phi)\qquad\textrm{as $\e\rightarrow 0$}.
\label{5.6}
\end{equation}
By using the second equation of  \eqref{3.1}we have that
\begin{align*}
\left|\int_{0}^{T}dt(\pe\dive\ue, \phi)\right| &\leq \sqrt{\e}\left|\int_{0}^{T}dt(\sqrt{\e}\pe,\pe\phi_{t})\right|\\
&\leq  \sqrt{\varepsilon}|| \sqrt{\varepsilon}\pe||_{{\PH}^{\frac{1}{2}+\frac{\beta}{2}}(H^{-\frac{1}{2}+\delta})}||\pe||_{H^{-\frac{1}{2}-\beta}({\PH}^{\frac{1}{2}-\delta})}.
\end{align*}
So by choosing $\delta$ and $\beta$ as in the Lemma 3.4 we have that
\begin{equation*}
\int_{0}^{T}dt(\pe\dive\ue, \phi)\rightarrow 0\qquad\textrm{as $\e\rightarrow 0$}.
\end{equation*}


 \end{document}